\documentclass[letterpaper, 10pt, conference]{ieeeconf}
\pdfoutput=1\IEEEoverridecommandlockouts 
\usepackage{amsmath,amssymb,url}
\usepackage{graphicx,subfigure,warmread}
\usepackage[all,import]{xy}
\usepackage{color}

\newcommand{\bracket}[1]{\ensuremath{\left[ #1 \right]}}
\newcommand{\braces}[1]{\ensuremath{\left\{ #1 \right\}}}
\newcommand{\parenth}[1]{\ensuremath{\left( #1 \right)}}

\newcommand{\refeqn}[1]{(\ref{eqn:#1})}
\newcommand{\reffig}[1]{Fig. \ref{fig:#1}}
\newcommand{\tr}[1]{\mathrm{tr}\ensuremath{\negthickspace\bracket{#1}}}
\newcommand{\trs}[1]{\mathrm{tr}\ensuremath{[#1]}}

\newcommand{\SO}{\ensuremath{\mathsf{SO(3)}}}
\newcommand{\T}{\ensuremath{\mathsf{T}}}
\renewcommand{\L}{\ensuremath{\mathsf{L}}}
\newcommand{\so}{\ensuremath{\mathfrak{so}(3)}}

\renewcommand{\Re}{\ensuremath{\mathbb{R}}}

\newcommand{\D}{\ensuremath{\mathbf{D}}}
\newcommand{\Sph}{\ensuremath{\mathsf{S}}}

\newcommand{\Q}{\ensuremath{\mathsf{Q}}}

\newcommand{\refprop}[1]{Proposition \ref{prop:#1}}

\title{\LARGE \bf
Geometric Tracking Control of\\ the Attitude Dynamics of a Rigid Body on \SO}

\author{Taeyoung Lee\authorrefmark{1}%
\thanks{Taeyoung Lee, Mechanical and Aerospace Engineering, Florida Institute of Technology, Melbourne, FL 39201 {\tt taeyoung@fit.edu}}%
\thanks{\textsuperscript{\footnotesize\ensuremath{*}}This research has been supported in part by NSF under grants CMMI-1029551.}
}

\newtheorem{prop}{Proposition}

\begin{document}
\allowdisplaybreaks
\maketitle \thispagestyle{empty} \pagestyle{empty}

\begin{abstract}
This paper provides new results for a tracking control of the attitude dynamics of a rigid body. Both of the attitude dynamics and the proposed control system are globally expressed on the special orthogonal group, to avoid complexities and ambiguities associated with other attitude representations such as Euler angles or quaternions. By selecting an attitude error function carefully, we show that the proposed control system guarantees a desirable tracking performance uniformly for nontrivial rotational maneuvers involving a large initial attitude error. In a special case where the desired attitude command is fixed, we also show that the attitude dynamics can be stabilized without the knowledge of an inertia matrix. These are illustrated by numerical examples.
\end{abstract}

\section{Introduction}

The attitude dynamics of a rigid body appears in various engineering applications, such as aerial and underwater vehicles, robotics, and spacecraft, and the attitude control problem has been extensively studied under various assumptions (see, for example, \cite{WieWeiJGCD89,WenKreITAC91,Sid97,Hug86}).

One of the distinct features of the attitude dynamics is that its configuration manifold is not linear: it evolves on a nonlinear manifold, referred as the special orthogonal group, $\SO$. This yields important and unique properties that cannot be observed from dynamic systems evolving on a linear space. For example, it has been shown that there exists no continuous feedback control system that asymptotically stabilizes an attitude globally on $\SO$~\cite{CroITAC84}. 

However, most of the prior work on the attitude control is based on minimal representations of an attitude, or quaternions. It is well known that any minimal attitude representations are defined only locally, and they exhibit kinematic singularities for large angle rotational maneuvers. Quaternions do not have singularities, but they have ambiguities in representing an attitude, as the three-sphere $\Sph^3$ double covers $\SO$. As a result, in a quaternion-based attitude control system, convergence to a single attitude implies convergence to either of the two disconnected, antipodal points on $\Sph^3$~\cite{MaySanPICDC09}. Therefore, a quaternion-based control system becomes discontinuous when applied to an actual attitude dynamics, and it may also exhibit unwinding behavior, where the controller unnecessarily rotates a rigid body through large angles~\cite{BhaBerSCL00}.

Geometric control is concerned with the development of control systems for dynamic systems evolving on nonlinear manifolds that cannot be globally identified with Euclidean spaces~\cite{Jur97,Blo03,BulLew05}. By characterizing geometric properties of nonlinear manifolds intrinsically, geometric control techniques completely avoids singularities and ambiguities that are associated with local coordinates or improper characterizations of a configuration manifold. This approach has been applied to fully actuated rigid body dynamics on Lie groups to achieve almost global asymptotic stability~\cite{BulLew05,MaiBerITAC06,CabCunPICDC08,ChaMcCITAC09,LeeLeoPICDC10}.

In this paper, we develop a geometric controller on $\SO$ to track an attitude and angular velocity command. The geometric attitude controllers studied in~\cite{BulLew05,ChaMcCITAC09,LeeLeoPICDC10} are not desirable in the sense that the magnitude of their control input converges to zero when the initial attitude error is maximized, i.e. the Eigen-axis rotation angle between the initial attitude and the initial command approaches $180^\circ$. This reduces the initial convergence rate significantly, and it destroys the unique advantage of geometric control approaches, namely effectiveness for large angle rotational maneuvers. 

The geometric tracking controller developed in this paper avoids this drawback by proposing a new configuration error function on $\SO$, and it exhibits a good tracking performance uniformly in large initial attitude errors. We also show that when the attitude tracking command is fixed, i.e. a stabilization problem, we can achieve exponential stability without the knowledge of an inertia matrix. For both cases, the region of attraction almost covers $\SO$, and the initial angular velocity error can be arbitrarily large, provided that a controller gain is sufficiently large. 

This paper is organized as follows. We present a global attitude dynamics model in Section \ref{sec:AD}. A new configuration error function and geometric control systems on $\SO$ are developed in Section III, followed by numerical results.

\section{Attitude Dynamics of a Rigid Body}\label{sec:AD}

We consider the rotational attitude dynamics of a fully-actuated rigid body. We define an inertial reference frame and a body fixed frame whose origin is located at the mass center of the rigid body. The configuration of the rigid body is the orientation of the body fixed frame with respect to the inertial frame, and it is represented by a rotation matrix $R\in\SO$, where the special orthogonal group $\SO$ is the group of $3\times 3$ orthogonal matrices with determinant of one, i.e., $\SO=\{R\in\Re^{3\times 3}\,|\,R^T R=I,\,\det{R}=1\}$.

The equations of motion are given by
\begin{gather}
J\dot \Omega + \Omega\times J\Omega = u,\label{eqn:Wdot}\\
\dot R = R\hat\Omega,\label{eqn:Rdot}
\end{gather}
where $J\in\Re^{3\times 3}$ is the inertia matrix in the body fixed frame, and $\Omega\in\Re^3$ and $u\in\Re^3$ are the angular velocity of the rigid body and the control moment, represented with respect to the body fixed frame, respectively.

The \textit{hat} map $\wedge :\Re^{3}\rightarrow\so$ transforms a vector in $\Re^3$ to a $3\times 3$ skew-symmetric matrix such that $\hat x y = x\times y$ for any $x,y\in\Re^3$. The inverse of the hat map is denoted by the \textit{vee} map $\vee:\so\rightarrow\Re^3$.
Several properties of the hat map are summarized as follows.
\begin{gather}
    \hat x y = x\times y = - y\times x = - \hat y x,\\
    \tr{A\hat x }=\frac{1}{2}\tr{\hat x (A-A^T)}=-x^T (A-A^T)^\vee,\label{eqn:hat1}\\
    \hat x  A+A^T\hat x=(\braces{\tr{A}I_{3\times 3}-A}x)^{\wedge},\label{eqn:xAAx}\\
R\hat x R^T = (Rx)^\wedge,\label{eqn:RxR}
\end{gather}
for any $x,y\in\Re^3$, $A\in\Re^{3\times 3}$, and $R\in\SO$.

\section{Geometric Tracking Control on $\SO$}

We develop a control system to follow a given smooth desired attitude command $R_d(t)\in\SO$. The kinematics equation for the attitude command can be written as
\begin{align}
\dot R_d = R_d \hat\Omega_d,\label{eqn:Rddot}
\end{align}
where $\Omega_d\in\Re^3$ is the desired angular velocity. 

\subsection{Attitude Error Function}

One of the important steps in constructing a control system on a nonlinear manifold $\Q$ is choosing a proper configuration error function, which is a smooth positive definite function $\Psi:\Q\times\Q\rightarrow\Re$ that measures the error between a current configuration and a desired configuration. Once a configuration error function is chosen, a configuration error vector, and a velocity error vector can be defined in the tangent $\T_q\Q$ by using the derivatives of $\Psi$~\cite{BulLew05}. Then, the remaining procedure is similar to nonlinear control system design in Euclidean spaces: control inputs are carefully designed as a function of these error vectors through a Lyapunov analysis on $\Q$, where a Lyapunov candidate also is written in terms of $\Psi$. Therefore, a configuration error function is critical in the design and analysis of a control systems on a manifold, and the corresponding performance and effectiveness of a control system directly depend on the choice of a configuration error function. 

But, the importance of a configuration error function has not been extensively studied in geometric controls, and  it is sometimes chosen without a careful consideration. Almost globally stabilizing controllers on $\SO$ have been studied in~\cite{BulLew05,ChaMcCITAC09}, where essentially, the following configuration error function is used to stabilize the attitude represented by the identity matrix:
\begin{align}
\Psi^\circ(I,R)=\frac{1}{2}\tr{I-R}.\label{eqn:Psi0}
\end{align}
This error function yields the following form of the configuration error vector $e_R^\circ=(R-R^T)^\vee\in\Re^3$ and the velocity error vector $e^\circ_\Omega = \Omega\in\Re^3$. A simple PD-type controller, i.e. $u^\circ=-k_R e^\circ_R - k_\Omega e^\circ_\Omega$ for positive constants $k_R,k_\Omega$, stabilizes the identity matrix $I$. This can be slightly generalized to achieve almost global stability.

However, this choice of a configuration error function is not desirable, since the magnitude of the corresponding configuration error vector $e^\circ_R$ is not proportional to the rotation angle about the Euler axis between the current attitude and the identity matrix: as the current attitude approaches to the opposite of the identity, i.e. $180^\circ$ rotation to $I$, the magnitude of the attitude error vector $\|e_R\|$ approaches zero. Therefore, the performance of this controller becomes worse as the initial attitude error becomes larger. This is not particularly desirable, since it destroys one of the distinct advantages of geometric controls of a rigid body, namely effectiveness for large angle rotational motions. 

In this paper, we introduce a new form of the configuration error function to avoid this drawback, and to improve tracking performances particularly for larger initial attitude errors. 

\begin{prop}\label{prop:1}
For a given tracking command $(R_d,\Omega_d)$, and current attitude and angular velocity $(R,\Omega)$, we define an attitude error function $\Psi:\SO\times\SO\rightarrow\Re$, an attitude error vector $e_R:\SO\times\SO\rightarrow\in\Re^3$, and an angular velocity error vector $e_\Omega:\SO\times\Re^3\times\SO\times\Re^3\rightarrow \Re^3$ as follows:
\begin{gather}
\Psi (R,R_d) = 2-\sqrt{1+\trs{R_d^T R}},\label{eqn:Psi}\\
e_R(R,R_d) =\frac{1}{2\sqrt{1+\tr{R_d^T R}}} (R_d^T R-R^TR_d)^\vee,\label{eqn:eR}\\
e_\Omega(R,\Omega,R_d,\Omega) = \Omega - R^T R_d\Omega_d.\label{eqn:eW}
\end{gather}
For a fixed $R_d$, the attitude error function $\Psi$ can be considered as a function of $R$ only. The attitude error vector $e_R$ is well defined in the sublevel set $L_2=\{R\in\SO\,|\, \Psi(R,R_d)< 2\}$. 

Then, the following statements hold:
\begin{itemize}
\item[(i)] $\Psi$ is positive definite about $R=R_d$.
\item[(ii)] in $L_2$, the left-trivialized derivative of $\Psi$ is given by
\begin{align}
\T^*_I \L_R\, (\D_R\Psi(R,R_d))= e_R.
\end{align}
\item[(iii)] the critical points of $\Psi$ are $\{R_d\}\cup\{R_d\exp (\pm\pi \hat s)\}$ for any $s\in\Sph^2$, and there exists only one critical point $\{R_d\}$ in $L_2$.
\item[(iv)] $\Psi$ is locally quadratic in $L_2$, since
\begin{align}
\|e_R\|^2 \leq \Psi(R,R_d) \leq 2\|e_R\|^2.\label{eqn:eRPsi}
\end{align}
\end{itemize}
\end{prop}

\begin{proof}
For any rotation matrix $Q=R_d^T R\in\SO$, its trace is bounded by $-1\leq\trs Q\leq 3$, and $\trs{Q}=3$ if and only if $Q=I$~\cite{ShuJAS93}. Substituting this into \refeqn{Psi}, it follows that $\Psi\geq 0$, and $\Psi=0$ if and only if $R=R_d$. This shows (i).

The infinitesimal variation of a rotation matrix can be written as
\begin{align*}
\delta R = \frac{d}{d\epsilon}\bigg|_{\epsilon=0} R \exp \epsilon\hat\eta = R\hat\eta
\end{align*}
for $\eta\in\Re^3$. Using this, the derivative of this error function with respect to $R$ is given by
\begin{align*}
\D_R \Psi(R,R_d)\cdot \delta R & = \frac{d}{d\epsilon}\bigg|_{\epsilon=0} \Psi(R\exp\epsilon\hat\eta,R_d)\\
& = -\frac{1}{2\sqrt{1+\trs{R_d^T R}}}\tr{R_d^T R\hat\eta}.
\end{align*}
This is well defined in $L_2$, since $\trs{R_d^TR} > -1$ in $L_2$. Using a property of the hat map given by \refeqn{hat1}, this can be written as
\begin{align*}
\D_R \Psi(R,R_d)\cdot R\hat\eta & = \frac{1}{2\sqrt{1+\trs{R_d^T R}}} (R_d^T R-R^TR_d)^\vee \cdot \eta\\
& = e_R\cdot \eta,
\end{align*}
which shows (ii). 

The critical points of $\Psi$ are the solutions $R\in\SO$ to the equation $R_d^TR-R^TR_d=0$ or $\trs{R_d^T R}=-1$, which are given by $R_d^TR=I$ or $R_d^TR=\exp(\pm\pi\hat s)$ for any $s\in\Sph^2$~\cite{BulLew05}. This shows the first part of (iii). From Rodrigues' formula, for any $Q=R_d^T R\in\SO$, there exists $x\in\Re^3$ with $\|x\|\leq \pi$ such that
\begin{align}
Q=\exp \hat x = I + \frac{\sin\|x\|}{\|x\|}\hat x+\frac{1-\cos\|x\|}{\|x\|^2}\hat x^2.\label{eqn:Q}
\end{align}
Substituting this into \refeqn{Psi}, we obtain
\begin{align*}
\Psi(R_d\exp\hat x, R_d) = 4\sin^2\frac{\|x\|}{4}.
\end{align*}
At the critical points $R=R_d\exp(\pm\pi\hat s)$ with $s\in\Sph^2$, the value of $\Psi$ becomes 2. This shows the second part of (iv). Substituting \refeqn{Q} into \refeqn{eR}, we obtain
\begin{align*}
\|e_R\|^2 = \sin^2 \frac{\|x\|}{2}=4\sin^2\frac{\|x\|}{4}\cos^2\frac{\|x\|}{4}.
\end{align*}
This shows (iv).
\end{proof}

The proposed attitude error function is more desirable than \refeqn{Psi0} in the sense that the magnitude of the attitude error vector $e_R$ is proportional to the rotation about the Euler axis between $R$ and $R_d$ (see \reffig{Psi}). This improves the tracking performance, especially for large angle rotational maneuvers with a large initial attitude error.

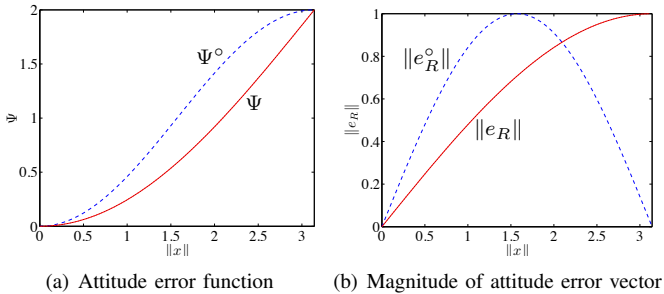
\begin{figure}[b]
\centerline{
\subfigure[Attitude error function]{
\renewcommand{\xyWARMinclude}[1]{\includegraphics[width=0.48\columnwidth]{#1}}
{\footnotesize\selectfont
$$\begin{xy}
\xyWARMprocessEPS{Psi}{pdf}
\xyMarkedImport{}
\xyMarkedMathPoints{1-15}
\end{xy}\vspace*{-0.35cm}
$$}}
\subfigure[Magnitude of attitude error vector]{
{\footnotesize\selectfont
\renewcommand{\xyWARMinclude}[1]{\includegraphics[width=0.48\columnwidth]{#1}}
$$\begin{xy}
\xyWARMprocessEPS{neR}{pdf}
\xyMarkedImport{}
\xyMarkedMathPoints{1-15}
\end{xy}\vspace*{-0.35cm}
$$}}}
\caption{Attitude error function $\Psi$ and the magnitude of the attitude error vector $\|e_R\|$ when $R_d^TR=\exp \hat x$, for $x/\|x\|=[1,0,0]$ and $\|x\|\in[0,\pi]$. For the attitude error function $\Psi^\circ$ used in other literatures (blue, dashed), $\|e_R^\circ\|$ is maximized when $\|x\|=\pi/2$, and it approaches $0$ as $\|x\|\rightarrow\pi$. This reduces the convergence rate of the corresponding control system significantly, when the initial attitude error approaches $180^\circ$. But, in the proposed attitude error function $\Psi$ (red), the magnitude of the attitude error vector $\|e_R\|$ is proportional to the rotation angle $\|x\|$ about the Euler axis between $R$ and $R_d$. This guarantees a good convergence rate uniformly in initial attitude errors.}\label{fig:Psi}
\end{figure}

\subsection{Attitude Error Dynamics}

We find the attitude error dynamics for the proposed attitude error function $\Psi$, the attitude error vector $e_R$, and the angular velocity error $e_\Omega$.
\begin{prop}\label{prop:2}
The error dynamics for $\Psi$, $e_R$, $e_\Omega$ satisfies
\begin{gather}
\frac{d}{dt}(\Psi(R,R_d))  = e_R\cdot e_\Omega,\label{eqn:Psidot}\\
\|\dot e_R\|  \leq \frac{1}{2} \|e_\Omega\|,\label{eqn:eRdot}\\
\dot e_\Omega  = J^{-1}(-\Omega\times J\Omega + u)+\hat\Omega R^T R_d\Omega_d- R^T R_d{\dot \Omega}_d.\label{eqn:eWdot}
\end{gather}
\end{prop}

\begin{proof}
Using the attitude kinematics equations \refeqn{Rdot}, \refeqn{Rddot}, the time derivative of the attitude error function is given by
\begin{align*}
\frac{d}{dt} \Psi &(R,R_d)  = -\frac{1}{2\sqrt{1+\trs{R_d^T R}}}\tr{R_d^T R\hat\Omega - \hat\Omega_d R_d^T R}\\
& = -\frac{1}{2\sqrt{1+\trs{R_d^T R}}}\tr{R_d^T R(\hat\Omega - R^T R_d\hat\Omega_d R_d^T R)},
\end{align*}
where we use a property of the hat map \refeqn{RxR}. Substituting \refeqn{eW} into this, and using \refeqn{hat1}, \refeqn{eR}, we obtain 
\begin{align*}
\frac{d}{dt} \Psi (R,R_d) &  = -\frac{1}{2\sqrt{1+\trs{R_d^T R}}}\tr{R_d^T R\hat e_\Omega}\\
& = \frac{1}{2\sqrt{1+\trs{R_d^T R}}}(R_d^T R-R^TR_d)^\vee \cdot e_\Omega,
\end{align*}
which shows \refeqn{Psidot}. Next, the time derivative of the attitude error vector is given by
\begin{align*}
\dot e_R
& = -\frac{\tr{-\hat\Omega_dR_d^T R+R_d^TR\hat\Omega}}{2(1+\tr{R_d^T R})} e_R+ \frac{1}{2\sqrt{1+\tr{R_d^T R}}}\\
&\quad\times (-\hat\Omega_dR_d^T R+R_d^T R\hat\Omega+\hat\Omega R^TR_d-R^T R_d\hat\Omega_d)^\vee.
\end{align*}
Using \refeqn{RxR}, \refeqn{eW}, this can be written in terms of $e_\Omega$ as
\begin{align*}
& \dot e_R  = -\frac{\tr{R_d^TR(\hat\Omega-R^T R_d\hat\Omega_dR_d^TR)}}{2(1+\tr{R_d^T R})} e_R+ \frac{1}{2\sqrt{1+\tr{R_d^T R}}}\\
&\times ( R_d^T R(\hat\Omega -R^T R_d\hat\Omega_dR_d^TR )
+(\hat\Omega -R^T R_d\hat\Omega_d R_d^T R)R^TR_d)^\vee\\
&\quad = -\frac{\tr{R_d^TR \hat e_\Omega}}{2(1+\tr{R_d^T R})} e_R+ \frac{1}{2\sqrt{1+\tr{R_d^T R}}}\\
&\quad\quad\times ( R_d^T R \hat e_\Omega
+\hat e_\Omega R^TR_d)^\vee.
\end{align*}
Using the properties of the hat map, given by \refeqn{hat1}, \refeqn{xAAx}, this can be further reduced to
\begin{align}
\dot e_R & = \frac{e_R\cdot e_\Omega}{\sqrt{1+\tr{R_d^T R}}} e_R+ \frac{1}{2\sqrt{1+\tr{R_d^T R}}}\nonumber\\
&\quad\times (\trs{R^T R_d}I -R^T R_d)e_\Omega\nonumber\\
& = \frac{1}{2\sqrt{1+\tr{R_d^T R}}} (\trs{R^T R_d}I -R^T R_d+2 e_R e_R^T)e_\Omega\nonumber\\
& \equiv E(R,R_d)e_\Omega,\label{eqn:eRdotE}
\end{align}
where $E(R,R_d)\in\Re^{3\times 3}$. From Rodrigues' formula, let $Q=R_d^TR=\exp\hat x\in\SO$ for $x\in\Re^3$. Using the Matlab Symbolic Computation Tool, the eigenvalues of $E(R,R_d)^TE(R,R_d)$ are given by
$\frac{1}{4},\frac{1}{4},\frac{1}{8}(1+\cos\|x\|)$. It follows that the matrix 2-norm of $E(R,R_d)$ is $\|E(R,R_d)\|=\frac{1}{2}$, which shows \refeqn{eRdot}.  

From \refeqn{Wdot}, \refeqn{Rdot}, \refeqn{Rddot}, and using the fact that $\hat\Omega_d\Omega_d=\Omega_d\times\Omega_d=0$ for any $\Omega_d\in\Re^3$, the time derivative of the angular velocity error $e_\Omega$ is given by
\begin{align*}
\dot e_\Omega & = \dot \Omega +\hat\Omega R^T R_d\Omega_d-R^T R_d\hat\Omega_d\Omega _d - R^T R_d{\dot \Omega}_d\\
& =J^{-1}(-\Omega\times J\Omega + u)+\hat\Omega R^T R_d\Omega_d- R^T R_d{\dot \Omega}_d,
\end{align*}
which shows \refeqn{eWdot}.
\end{proof}

\subsection{Attitude Tracking}

Here we define a control system to follow a given attitude command, and we show exponential stability. 

\begin{prop}
For a given attitude command $R_d(t)$, and positive constants $k_R,k_\Omega\in\Re$, we define a control input $u\in\Re^3$ as follows:
\begin{align}
u = -k_R e_R - k_\Omega e_\Omega + \Omega\times J\Omega -J (\hat\Omega R^T R_d\Omega_d- R^T R_d{\dot \Omega}_d).\label{eqn:u}
\end{align}
This control system stabilizes the zero equilibrium of the tracking error $e_R,e_\Omega$ exponentially, and an estimation of the region of attraction is given by
\begin{gather}
\Psi(R(0),R_d(0)) < 2,\label{eqn:Psi0}\\
\|e_\Omega(0)\|^2 < \frac{2}{\lambda_{\max}(J)} k_R \{2-\Psi(R(0),R_d(0))\},\label{eqn:eW0}
\end{gather}
where $\lambda_{\max}(J)$ denotes the maximum eigenvalue of the inertia matrix $J$.
\end{prop}

\begin{proof}
We first show that the sublevel set $L_2=\{R\in\SO\,|\,\Psi(R,R_d)<2\}$ is a positively invariant set. Consider the following Lyapunov function
\begin{align*}
\mathcal{W} = \frac{1}{2} e_\Omega \cdot J e_\Omega + k_R \Psi(R,R_d).
\end{align*}
According to \refeqn{eRPsi}, this is locally positive definite. Substituting \refeqn{u} into \refeqn{eWdot}, we obtain
\begin{align}
J\dot e_\Omega = -k_R e_R -k_\Omega e_\Omega.\label{eqn:JeWdot}
\end{align}
Using \refeqn{Psidot}, \refeqn{JeWdot}, the time-derivative of $\mathcal{W}$ is given by
\begin{align*}
\dot{\mathcal{W}} & = e_\Omega\cdot(-k_R e_R -k_\Omega e_\Omega) + k_R e_R\cdot e_\Omega \\
& = - k_R \|e_\Omega\|^2 \leq 0. 
\end{align*}
This implies that $\mathcal{W}(t)\leq \mathcal{W}(0)$ for all $t>0$, and from \refeqn{eW0}, we have
\begin{align*}
\mathcal{W}(0) \leq \frac{1}{2}\lambda_{\max}(J)\|e_\Omega(0)\|^2 + k_R\Psi(R(0),R_d(0)) < 2k_R.
\end{align*}
Therefore, we obtain
\begin{align*}
k_R\Psi(R(t),R_d(t))\leq \mathcal{W}(t)\leq \mathcal{W}(0) < 2k_R
\end{align*}
for all $t>0$. This follows that $\Psi(R(t),R_d(t))<2$ always. So, the sublevel set $L_2=\{R\in\SO\,|\,\Psi(R,R_d)<2\}$ is a positively invariant set under the given assumptions. Then, according to \refprop{1}, the attitude error vector $e_R$ is well defined, and there exists only one critical point of $\Psi$ in $L_2$.

To show exponential stability, we consider the following Lyapunov function:
\begin{align*}
\mathcal{V} = \frac{1}{2} e_\Omega \cdot J e_\Omega + k_R \Psi(R,R_d) + c_2 e_\Omega\cdot e_R
\end{align*}
for a positive constant $c_2$. From \refeqn{eRPsi}, we obtain
\begin{align}
z^T W_{11} z \leq \mathcal{V} \leq z^T W_{12} z,\label{eqn:Vb}
\end{align}
where $z = [\|e_R\|;\,\|e_\Omega\|]\in\Re^2$, and the matrices $W_{11},W_{12}\in\Re^{2\times 2}$ are given by
\begin{align*}
W_{11} = \begin{bmatrix}  k_R & \frac{1}{2}c_2\\
\frac{1}{2}c_2 & \frac{1}{2}\lambda_{\min}(J)
\end{bmatrix},\quad
W_{12} = \begin{bmatrix}  2k_R & \frac{1}{2}c_2\\
\frac{1}{2}c_2 & \frac{1}{2}\lambda_{\max}(J)
\end{bmatrix}.
\end{align*}

From \refeqn{JeWdot}, \refeqn{Psidot}, \refeqn{eRdotE}, the time derivative of the Lyapunov function $\mathcal{V}$ along the solution of the controlled system is given by
\begin{align*}
\dot{\mathcal{V}} 
& = e_\Omega\cdot J\dot e_\Omega + k_R \dot\Psi + c_2 \dot e_\Omega\cdot e_R + c_2 e_\Omega \cdot \dot e_R\\
& = e_\Omega \cdot ( -k_\Omega e_\Omega - k_R e_R) + k_R e_R\cdot e_\Omega\\
&\quad  + c_2 (J^{-1}( - k_\Omega e_\Omega - k_R e_R)) \cdot e_R + c_2 e_\Omega \cdot E_R(R,R_d)e_\Omega.
\end{align*}
Using \refeqn{eRdot}, this is bounded by
\begin{align}
\dot{\mathcal{V}} 
& \leq - \parenth{k_\Omega - \frac{c_2}{2}}\|e_\Omega\|^2
- \frac{c_2 k_R}{\lambda_{\max}(J)} \|e_R\|^2\nonumber\\
&\quad+ \frac{c_2 k_\Omega}{\lambda_{\min}(J)} \|e_R\|\|e_\Omega\|\nonumber\\
& = -z^T W_2 z,\label{eqn:Vdot}
\end{align}
where the matrix $W_2\in\Re^{2\times 2}$ is given by
\begin{align}
W_2 = \begin{bmatrix}  \frac{c_2 k_R}{\lambda_{\max}(J)} & -\frac{c_2 k_\Omega}{2\lambda_{\min}(J)}\\
- \frac{c_2 k_\Omega}{2\lambda_{\min}(J)} & k_\Omega - \frac{c_2}{2}\end{bmatrix}.\label{eqn:W2}
\end{align}
We choose the positive constant $c_2$ such that
\begin{align*}
c_2 < \min\braces{\sqrt{2k_R\lambda_{\min}(J)},\, 2k_\Omega,\,
\frac{4 k_R k_\Omega \lambda_{\min}^2(J)}{2 k_R \lambda_{\min}^2(J)+ k_\Omega^2\lambda_{\max}(J) }}.
\end{align*}
Then, the matrices $W_{11},W_{12},W_{2}$ become positive definite, which implies that $\mathcal{V}$ is quadratic, and 
\begin{align}
\mathcal{V}(t)\leq \mathcal{V}(0) \exp \parenth{-\frac{\lambda_{\min}(W_{2})}{\lambda_{\max}(W_{12})}t}.\label{eqn:Vt}
\end{align}
Therefore, the zero equilibrium of the attitude and the angular velocity tracking error $(e_R,e_\Omega)$ is exponentially stable. 
\end{proof}

In this paper, we claim that this control system stabilizes the zero equilibrium of the attitude and angular velocity tracking errors \textit{almost semi-globally} in the sense that the region of attraction given by \refeqn{Psi0}, \refeqn{eW0} satisfies the following properties: the initial attitude region given by \refeqn{Psi0} almost cover $\SO$, since it only excludes the two-dimensional subset $\{R_d(0)\exp(\pm\pi\hat s),\; s\in\Sph^2\}$ from the three-dimensional $\SO$; the initial angular velocity could be arbitrarily large by choosing a sufficiently larger gain $k_R$ in \refeqn{eW0}. 

\subsection{Attitude Stabilization Without Inertia Matrix} 

The proposed attitude tracking controller requires the exact value of the inertia matrix. Here, we show that in a special case, where the desired attitude is fixed, i.e. $\Omega_d(t)\equiv 0$, we can stabilize the attitude error without the knowledge of the inertia matrix. 

\begin{prop}
Suppose that the desired attitude $R_d$ is fixed so that $\Omega_d(t)\equiv 0$ for any $t>0$. For positive constant $k_R,k_\Omega\in\Re$, we define a control input $u'\in\Re^3$ as follows:
\begin{align}
u'= -k_R e_R -k_\Omega e_\Omega.\label{eqn:up}
\end{align}
This control system stabilizes the zero equilibrium of the errors $e_R,e_\Omega$ exponentially, and an estimation of the region of attraction is given by
\begin{gather}
\Psi(R(0),R_d(0)) < 2,\label{eqn:Psi0}\\
\|e_\Omega(0)\|^2 < \frac{2}{\lambda_{\max}(J)} k_R \{2-\Psi(R(0),R_d(0))\},\label{eqn:eW0}
\end{gather}
where $\lambda_{\max}(J)$ denotes the maximum eigenvalue of the inertia matrix $J$.
\end{prop}
\begin{proof}
Similar to the proof of \refprop{2}, we first show that the sublevel set $L_2=\{R\in\SO\,|\,\Psi(R,R_d)<2\}$ is a positively invariant set. Consider the following Lyapunov function
\begin{align*}
\mathcal{W}' = \frac{1}{2} e_\Omega \cdot J e_\Omega + k_R \Psi(R,R_d).
\end{align*}
According to \refeqn{eRPsi}, this is locally positive definite. Substituting \refeqn{up} into \refeqn{eWdot}, we obtain
\begin{align}
J\dot e_\Omega = -k_R e_R -k_\Omega e_\Omega-\Omega\times J\Omega.\label{eqn:JeWdotp}
\end{align}
Using \refeqn{Psidot}, \refeqn{JeWdotp}, the time-derivative of $\mathcal{W}'$ is given by
\begin{align*}
\dot{\mathcal{W}} & = e_\Omega\cdot(-k_R e_R -k_\Omega e_\Omega-\Omega\times J\Omega) + k_R e_R\cdot e_\Omega.
\end{align*}
According to \refeqn{eW}, we have $e_\Omega=\Omega$ when $\Omega_d=0$. Thus, $e_\Omega\cdot(\Omega\times J\Omega) = e_\Omega\cdot(e_\Omega\times Je_\Omega) =0$. Then, this reduces to 
\begin{align*}
\dot{\mathcal{W}} & = - k_R \|e_\Omega\|^2 \leq 0. 
\end{align*}
Similar to the proof of \refprop{1}, this implies that the sublevel set $L_2=\{R\in\SO\,|\,\Psi(R,R_d)<2\}$ is a positively invariant set under the given assumptions.

To show exponential stability, we consider the following Lyapunov function:
\begin{align*}
\mathcal{V}' = \frac{1}{2} e_\Omega \cdot J e_\Omega + k_R \Psi(R,R_d) + c_2 e_\Omega\cdot e_R
\end{align*}
for a positive constant $c_2$. This satisfies \refeqn{Vb}. 

From \refeqn{JeWdotp}, \refeqn{Psidot}, \refeqn{eRdotE}, the time derivative of the Lyapunov function $\mathcal{V}]$ along the solution of the controlled system is given by
\begin{align*}
\dot{\mathcal{V}}' 
& = e_\Omega \cdot ( -k_\Omega e_\Omega - k_R e_R-\Omega\times J\Omega) + k_R e_R\cdot e_\Omega\\
&\quad  + c_2 (J^{-1}( - k_\Omega e_\Omega - k_R e_R-\Omega\times J\Omega)) \cdot e_R\\
&\quad + c_2 e_\Omega \cdot E_R(R,R_d)e_\Omega.
\end{align*}
Since $\Psi(R(t),R_d)< 2$, we have $\|e_R\|< \sqrt{2}$ from \refeqn{eRPsi}. We also have $e_\Omega=\Omega$ since $\Omega_d=0$. Then, the following inequality is satisfied:
\begin{align*}
\|c_2 (J^{-1}(\Omega\times J\Omega))\cdot e_R\| \leq \sqrt{2}c_2 \frac{\lambda_{\max}(J)}{\lambda_{\min}(J)} \|e_\Omega\|^2.
\end{align*}
Then, similar to \refeqn{Vdot}, we obtain
\begin{align*}
\dot{\mathcal{V}}' & \leq -z^T W'_2 z,
\end{align*}
where the matrix $W'_2\in\Re^{2\times 2}$ is given by
\begin{align}
W'_2 = \begin{bmatrix}  \frac{c_2 k_R}{\lambda_{\max}(J)} & -\frac{c_2 k_\Omega}{2\lambda_{\min}(J)}\\
- \frac{c_2 k_\Omega}{2\lambda_{\min}(J)} & k_\Omega - \alpha c_2\end{bmatrix},\label{eqn:W2p}
\end{align}
where $\alpha=\frac{1}{2}+\sqrt{2}\frac{\lambda_{\max}(J)}{\lambda_{\min}(J)}$.

We choose the positive constant $c_2$ such that
\begin{align*}
c_2 < \min\braces{\sqrt{2k_R\lambda_{\min}(J)},\, \frac{k_\Omega}{\alpha},\,
\frac{4 k_R k_\Omega \lambda_{\min}^2(J)}{4\alpha k_R \lambda_{\min}^2(J)+ k_\Omega^2\lambda_{\max}(J) }}.
\end{align*}
Then, the matrices $W_{11},W_{12},W'_{2}$ become positive definite, which implies the zero equilibrium of the attitude and the angular velocity error $(e_R,e_\Omega)$ is exponentially stable. 
\end{proof}

This control system allows us to stabilize a fixed attitude without the knowledge of the inertia matrix, since the control input \refeqn{up} is independent of $J$. But, this reduces the convergence rate. As discussed in \refeqn{Vt}, the convergence rate of the controlled system depends on the eigenvalue of the matrix $W_2$. Comparing \refeqn{W2p} with \refeqn{W2}, we expect that the eigenvalues of $W'_2$ are less than the eigenvalues of $W_2$ since $\alpha>\frac{1}{2}$. 

\subsection{Properties}

One of the unique properties of the presented controller is that it is directly developed on $\SO$ using rotation matrices. Therefore, it avoids the complexities and singularities associated with local coordinates of $\SO$, such as Euler angles. It also avoids the ambiguities that arise when using quaternions to represent the attitude dynamics. As the three-sphere $\Sph^3$ double covers $\SO$, any attitude feedback controller designed in terms of quaternions could yield different control inputs depending on the choice of quaternion vectors. The corresponding stability analysis would need to carefully consider the fact that convergence to a single attitude implies convergence to either of the two disconnected, antipodal points on $\Sph^3$~\cite{MaySanPICDC09}. This requires a continuous selection of the sign of quaternions or a discontinuous control system, which are shown to be sensitive to small measurement noise~\cite{SanMesPACC06}. Without these considerations, a quaternion-based controller can exhibit an unwinding phenomenon, where the controller unnecessarily rotates the attitude through large angles~\cite{BhaBerSCL00}. In this paper, the use of rotation matrices in the controller design and stability analysis completely eliminates these difficulties.

Another novelty of the presented controller is the choice of the attitude error function in \refeqn{Psi}. It is carefully designed to guarantee a good tracking performance for large attitude error. In contrast to other attitude control systems on $\SO$ constructed by \refeqn{Psi0}~\cite{BulLew05,ChaMcCITAC09}, the magnitude of the attitude error vector is proportional to the value of the attitude error function such that the corresponding control system is uniformly effective for larger attitude errors. These are illustrated by numerical examples in the next section.

\section{Numerical Examples}

\begin{figure}
\centerline{
	\subfigure[Attitude error function $\Psi$]{
		\includegraphics[width=0.5\columnwidth]{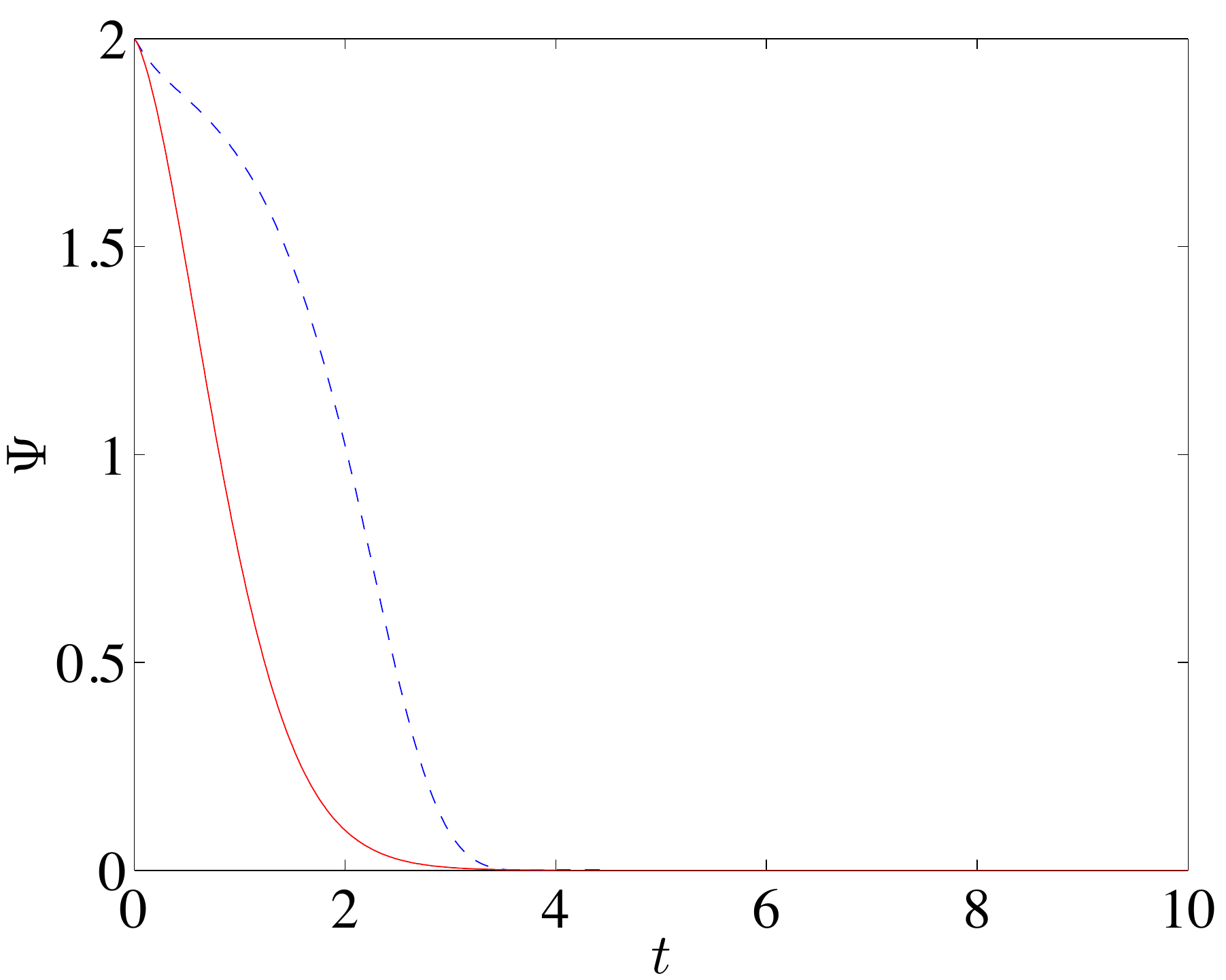}}
	\subfigure[Magnitude of angular velocity error $\|e_\Omega\|$ ($\mathrm{rad/sec}$)]{
		\includegraphics[width=0.5\columnwidth]{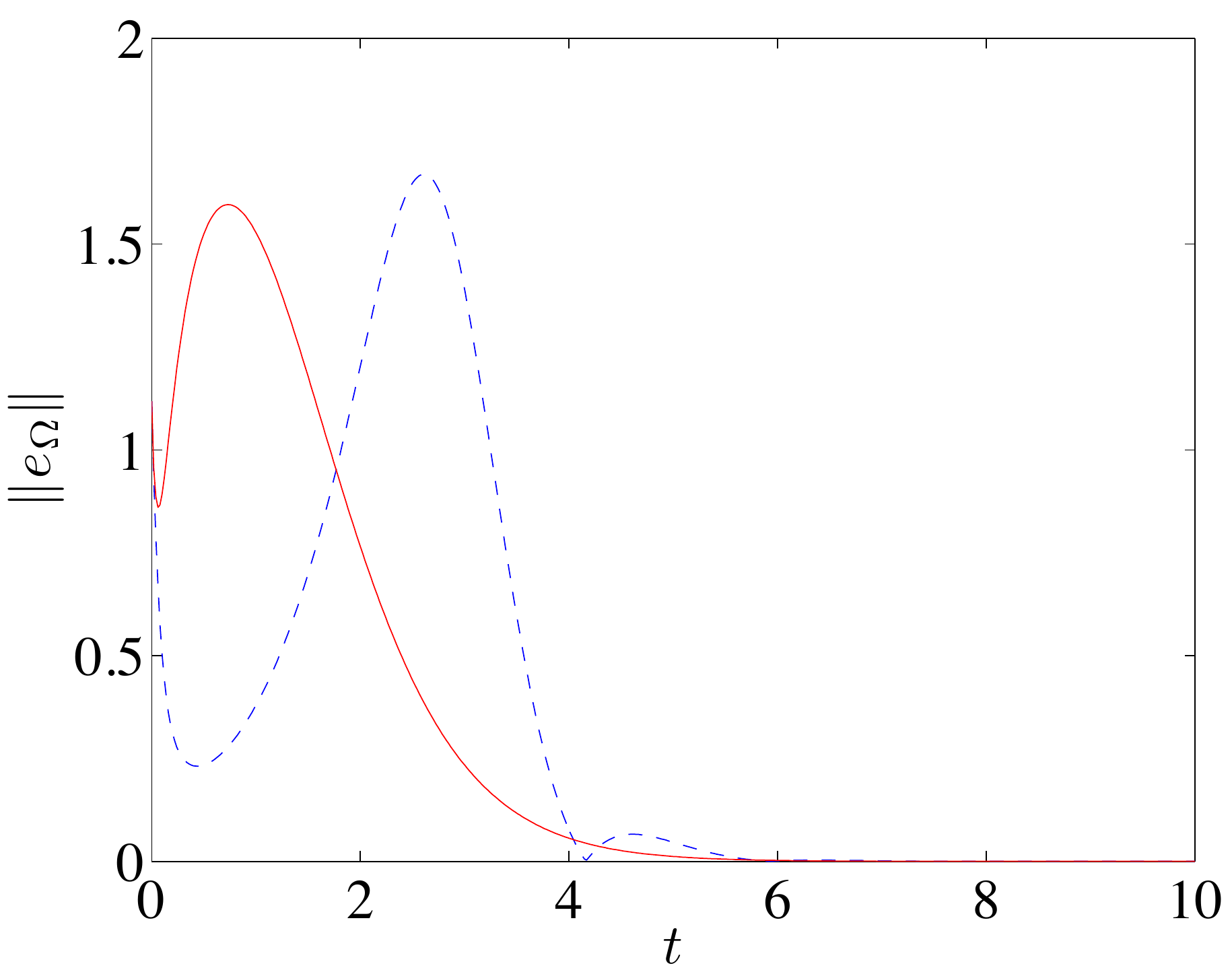}}
}
\centerline{
	\subfigure[Angular velocity $\Omega$ ($\mathrm{rad/sec}$)]{
		\includegraphics[width=0.5\columnwidth]{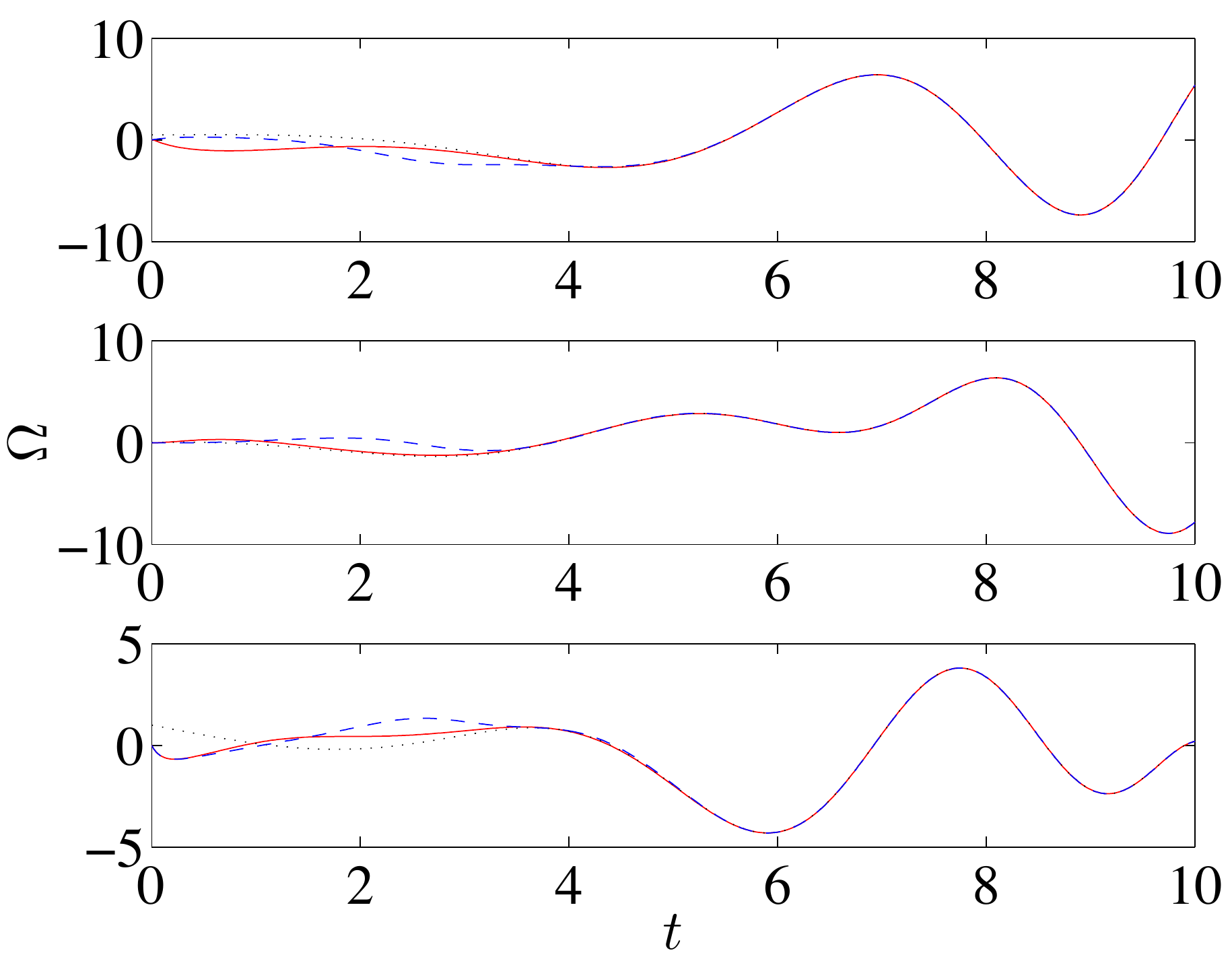}}
	\subfigure[Control input $u$ ($\mathrm{Nm}$)]{
		\includegraphics[width=0.5\columnwidth]{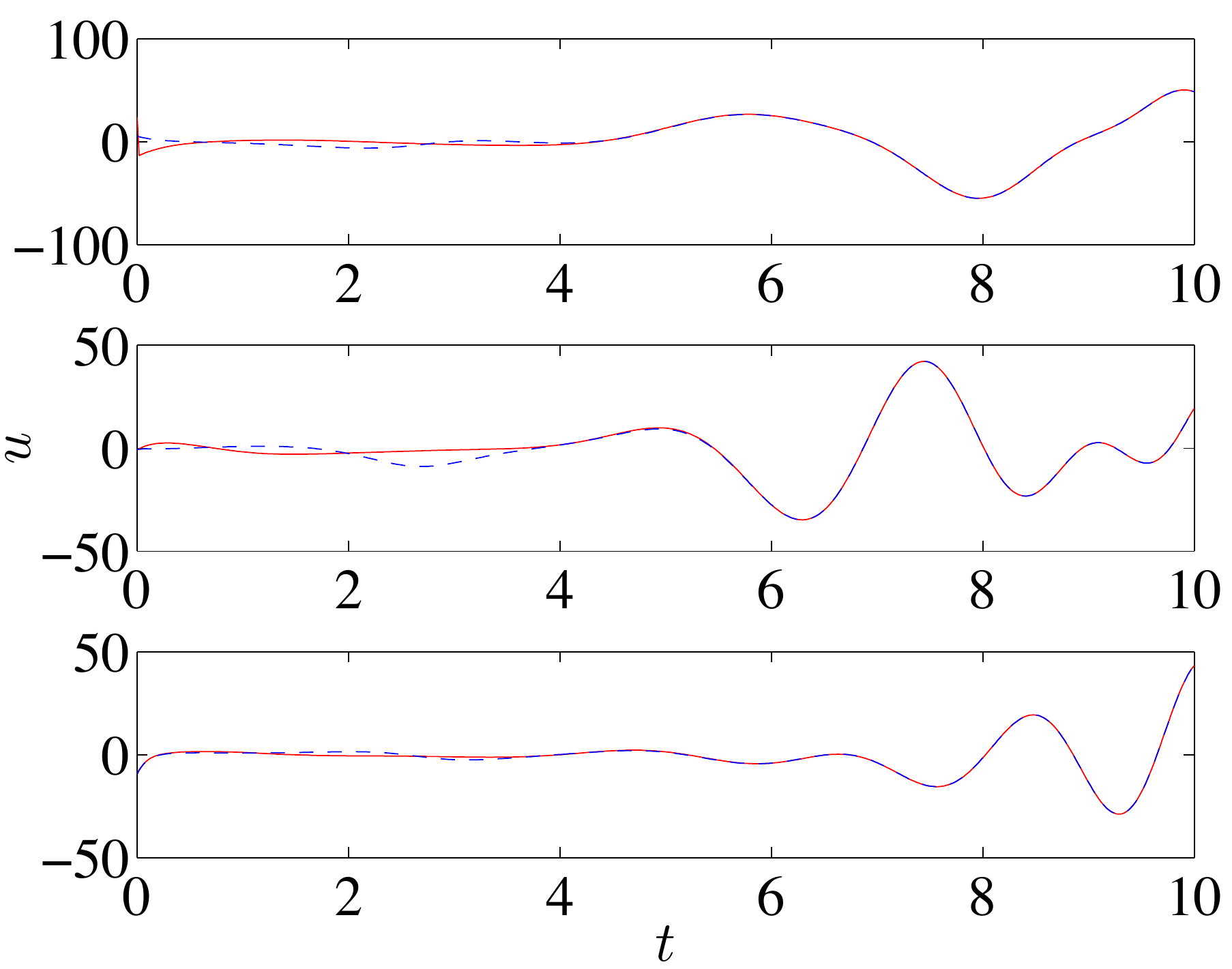}}
}
\caption{Attitude tracking (proposed controller: red,solid,\quad control system constructed by \refeqn{Psi0}: blue,dashed,\quad command:black,dotted)}\label{fig:1}
\end{figure}

We choose the inertia matrix of a rigid body and initial conditions as follows:
\begin{gather*}
J=\mathrm{diag}[3,\,2,\,1]\,\mathrm{kgm^2}\\
R(0)=I,\quad \Omega(0)=[0,\, 0,\, 0]\,\mathrm{rad/sec}
\end{gather*}
We consider two cases:
\begin{itemize}
\item[(i)] Attitude tracking with the full knowledge of $J$. The desired attitude command is described by using 3-2-1 Euler angles ~\cite{ShuJAS93}, i.e. $R_d(t)=R_d(\phi(t),\theta(t),\psi(t))$, and these angles are chosen as
\begin{align*}\hspace*{-0.5cm}
\phi(t) & =  0.999\pi+0.5t,\;\theta(t)= 0.1t^2,\; \psi(t)=-0.2t+0.5t^2,
\end{align*}
where the unit of these angles is radian, when the simulation time $t$ is in seconds. The corresponding angular velocity command $\Omega(t)$, and its time-derivative $\dot\Omega(t)$ are obtained from the attitude kinematics equation \refeqn{Rddot}.
\item[(ii)] Attitude stabilization without the knowledge of $J$. The desired attitude is chosen as
\begin{align*}
R_d = \exp (0.999\pi \hat s),\quad \text{where $s=\frac{1}{\sqrt{3}}[1,-1,1]$}.
\end{align*}
Since $R_d$ is fixed, we have $\Omega_d(t)=\dot\Omega_d(t)=0$. 
\end{itemize}
We use the control system \refeqn{u} for the first case, and we use the control system \refeqn{up} for the second case. For both cases, the controller gains are chosen as
\begin{align*}
k_R = 12,\quad k_\Omega = 8.4.
\end{align*}
Note that the desired attitude command of the first case represents a nontrivial rotational maneuver, and the initial attitude error of both cases is $0.999\pi=179.82^\circ$ in terms of the rotation angle about the Euler axis between $R(0)$ and $R_d(0)$.

It has been shown that general-purpose numerical integrators fail to preserve the structure of the special orthogonal group $\SO$, and they may yields unreliable computational results for complex maneuvers of rigid bodies~\cite{IseMunAN00,HaiLub00}. In this paper, we use a geometric numerical integrators, referred to as a Lie group variational integrator, to preserve the underlying geometric structures of the attitude dynamics accurately~\cite{LeeLeoCMDA07}.

Simulation results are represented in the following figures, where the responses of the proposed control system (red, solid lines) are compared with a control system based on \refeqn{Psi0} in~\cite{BulLew05,ChaMcCITAC09} (blue, dashed lines). At \reffig{Psi}, we showed that the control system based on \refeqn{Psi0} yields a small control input when the initial attitude error is close to $180^\circ$. These are observed again in the subfigure (d) for both cases. As a result, the initial convergence rates of the attitude error and the angular velocity error are relatively poor in the subfigures (a) and (b): it takes a longer time to converge in blue, dashed lines. But, the proposed control system exhibits more desirable convergence properties for a given complex rotational maneuvers involving large initial attitude errors. 

\begin{figure}
\centerline{
	\subfigure[Attitude error function $\Psi$]{
		\includegraphics[width=0.5\columnwidth]{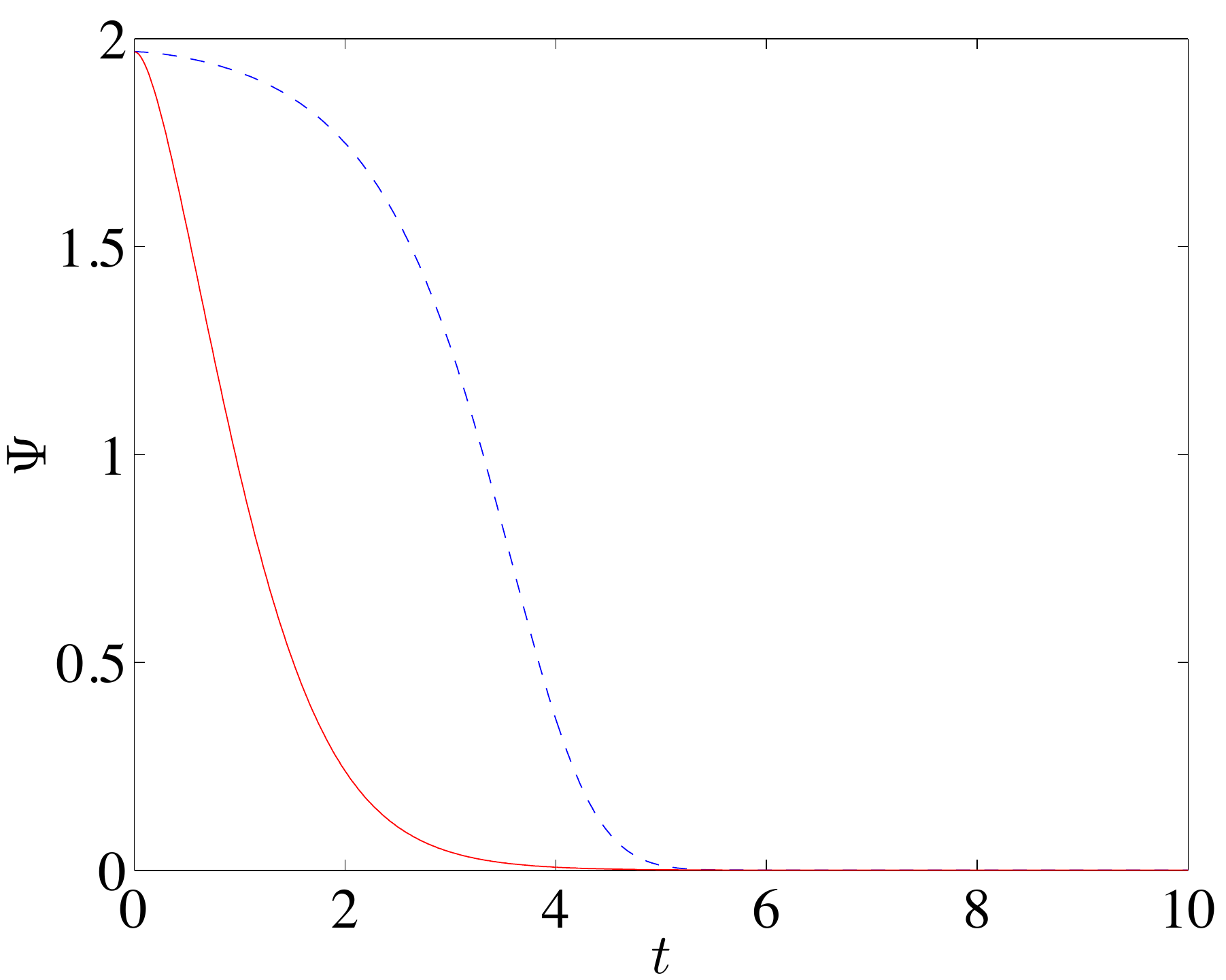}}
	\subfigure[Magnitude of angular velocity error $\|e_\Omega\|$ ($\mathrm{rad/sec}$)]{
		\includegraphics[width=0.5\columnwidth]{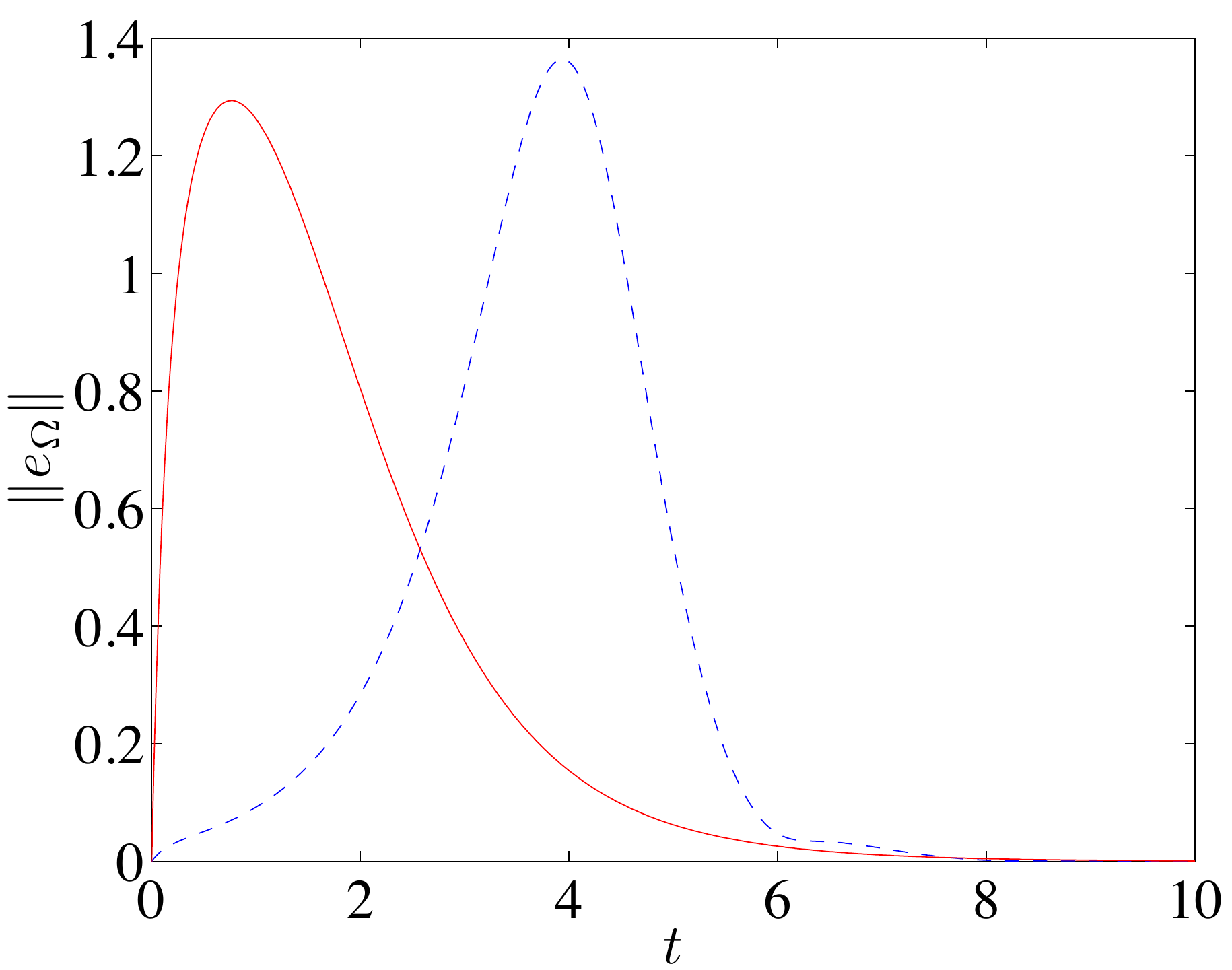}}
}
\centerline{
	\subfigure[Angular velocity $\Omega$ ($\mathrm{rad/sec}$)]{
		\includegraphics[width=0.5\columnwidth]{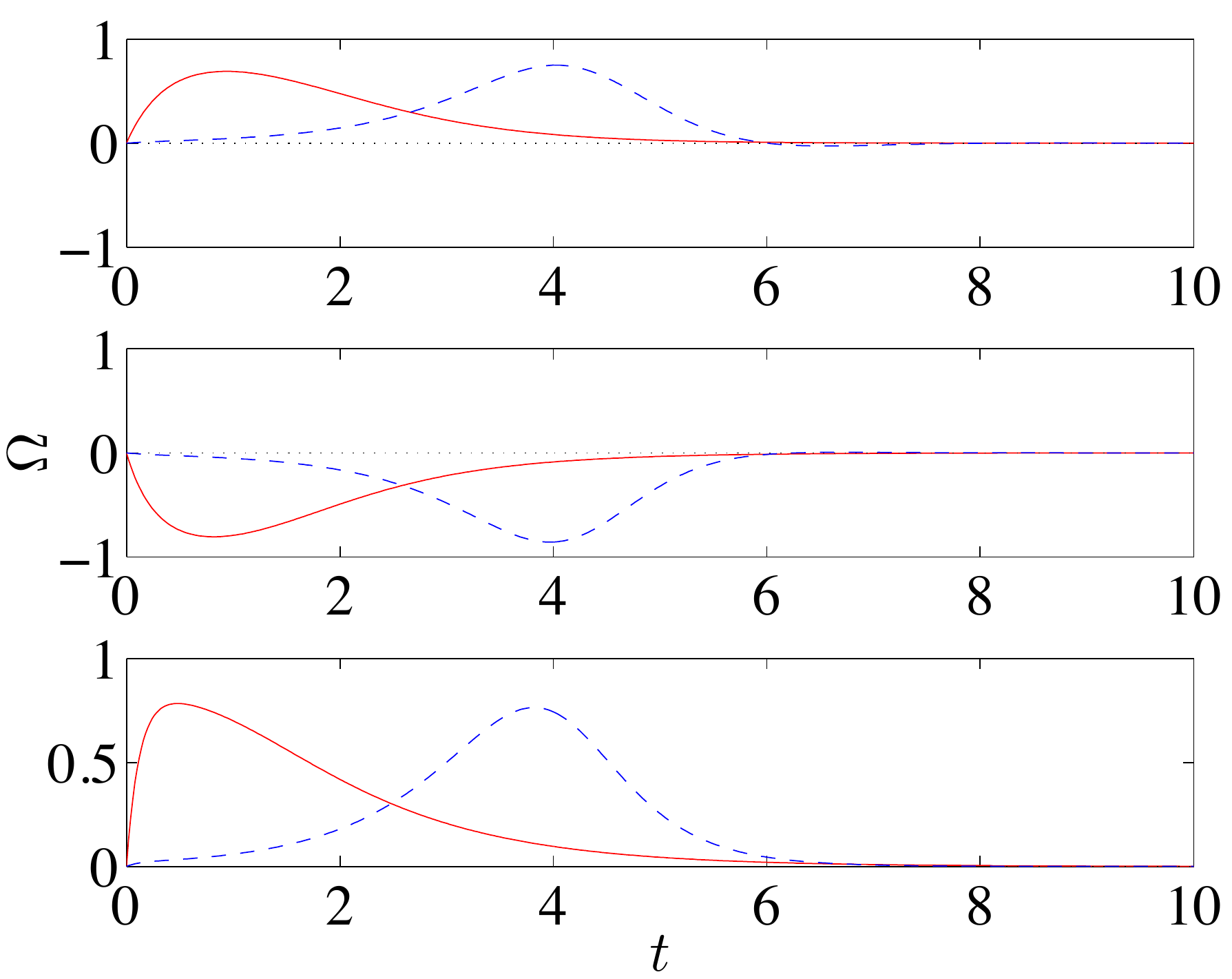}}
	\subfigure[Control input $u$ ($\mathrm{Nm}$)]{
		\includegraphics[width=0.5\columnwidth]{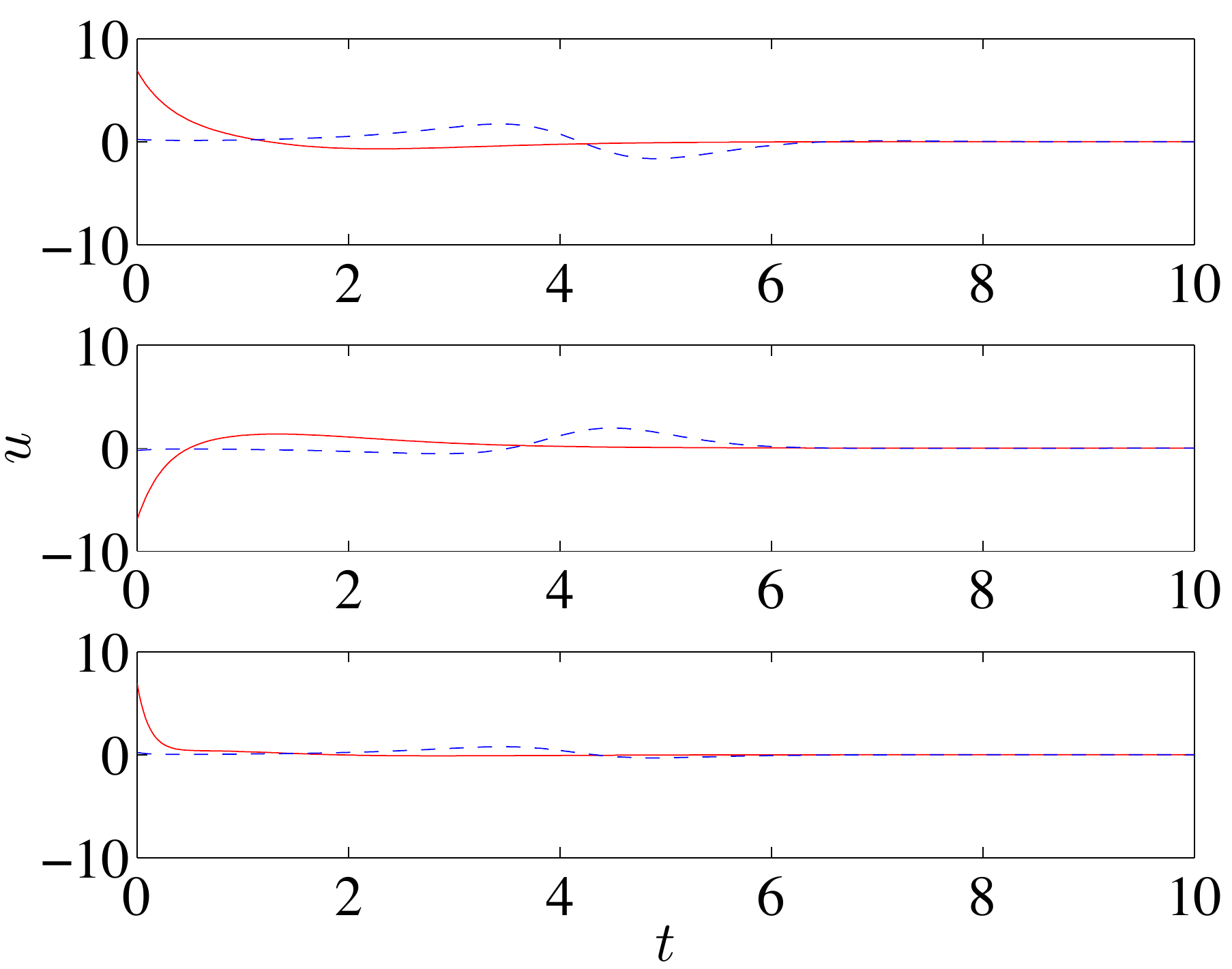}}
}
\caption{Attitude stabilizing without the knowledge of the inertia matrix. (proposed controller: red,solid,\quad control system constructed by \refeqn{Psi0}: blue,dashed,\quad command:black,dotted)}
\label{fig:2}
\end{figure}

\section{Conclusion}

We have developed a geometric tracking control system on $\SO$. The proposed control system is constructed directly on $\SO$ to avoid singularities and ambiguities that are inherent to other attitude representations, and its tracking performance is guaranteed uniformly in large initial attitude errors. We also show that in a special case where the desired attitude command is fixed, the proposed control system does not require the full knowledge of an inertia matrix.

\bibliography{ACC11.2}
\bibliographystyle{IEEEtran}

\end{document}